\documentclass[11pt]{amsart}

\usepackage{amsmath,amssymb,amsthm}
\usepackage{fullpage}
\usepackage[all]{xy}
\usepackage{color}

 \DeclareMathOperator{\tor}{Tor}

\DeclareMathOperator{\Span}{Span}

\newcommand{\col}{\,{:}\,}
\newcommand{\tth}{^{\operatorname{th}}}

\theoremstyle{plain}
\newtheorem{thm}{Theorem}
\newtheorem{lem}[thm]{Lemma}
\newtheorem{prop}[thm]{Proposition}

\theoremstyle{definition}
\newtheorem{defn}[thm]{Definition}

\newtheorem*{exmp}{Example}
\newtheorem*{rem}{Remark}
\theoremstyle{remark}

\def\P{\mathbb{P}}
\def\A{\mathbb{A}}

\def\N{\mathbb{N}}
\def\C{\mathbb{C}}

\begin{document}
    \title{Effectivity of Dynatomic cycles for morphisms of projective varieties using deformation theory}
    \author[Hutz]{Benjamin Hutz}
    \address{The Graduate Center, The City University of New York, New York, NY 10016 USA}
    \email{bhutz@gc.cuny.edu}

    \thanks{The author would like to thank Dan Abramovich for suggesting the method, Bjorn Poonen for several fruitful discussions on the topic, and Xander Faber for reading an earlier version.}




    \begin{abstract}
         Given an endomorphism of a projective variety, by intersecting the graph and the diagonal varieties we can determine the set of periodic points.  In an effort to determine the periodic points of a given minimal period, we follow a construction similar to cyclotomic polynomials.  The resulting zero-cycle is called a dynatomic cycle and the points in its support are called formal periodic points.  This article gives a proof of the effectivity of dynatomic cycles for morphisms of projective varieties using methods from deformation theory.
    \end{abstract}

\maketitle

     \section{Introduction}
        Consider an analytic function $\phi:\C^N \to \C^N$ given by
        \begin{equation*}
            [z_1,\ldots,z_N] \mapsto [\phi_1(z_1,\ldots,z_N),\ldots,\phi_N(z_1,\ldots,z_N)].
        \end{equation*}
        We can iterate the function $\phi$ to create a (discrete) dynamical system and denote the \emph{$n\tth$ iterate} as $\phi^n = \phi(\phi^{n-1})$.  The \emph{periodic points} of $\phi$ are the points $P \in \C^N$ such that $\phi^n(P) = P$ for some integer $n$.  We call $n$ the \emph{period} of $P$ and the least such $n$ the \emph{minimal} period of $P$.  Denote the coordinate functions of the $n\tth$ iterate as $\phi^n=[\phi_1^n,\ldots,\phi^n_N]$.  The set of periodic points of period $n$, but not necessarily minimal period $n$, for $\phi$ is the set of solutions to the system of equations
        \begin{equation*}
            \phi_i^n(z_1,\ldots,z_N) = z_i \quad \text{for } 1 \leq i \leq N.
        \end{equation*}
        To find the points of minimal period $n$, we could attempt to remove the points of period strictly less than $n$ from this set.  In the case of $\phi(z) \in \C[z]$, we can do this through division, as with cyclotomic polynomials. Consider the zeros of
        \begin{equation} \label{eq1}
            \prod_{d \mid n} (\phi^d(z)-z)^{\mu\left(\frac{n}{d}\right)}
        \end{equation}
        where $\mu$ is the M\"{o}bius function defined as $\mu(1) = 1$ and
        \begin{equation*}
            \mu(n) = \begin{cases}
              (-1)^{\omega} & n \text{ is square-free with $\omega$ distinct prime factors}\\
              0 & n \text{ is not square-free}.
            \end{cases}
        \end{equation*}
         Two fundamental questions come to mind.  First, are the zeros of the resulting function exactly the set of periodic points of minimal period $n$? Unfortunately, the answer is no.  For example for $\phi(z) = z^2 - \frac{3}{4}$ and $n=2$ we have
         \begin{equation*}
            \frac{\phi^2(z)-z}{\phi(z)-z} = \left(z+\frac{1}{2}\right)^2,
         \end{equation*}
         but that $-\frac{1}{2}$ is a fixed point.  However, it is true that all of the points of minimal period $n$ are among the zeros, demonstrated for single variables polynomials in \cite[Theorem 2.4]{Morton} and for morphisms of projective varieties by the author in \cite[Proposition 4.1]{Hutz1}.  Secondly, does the resulting function have poles as well as zeros?  Fortunately, the answer is no, demonstrated in the single variable polynomial case in \cite[Theorem 2.5]{Morton}, for automorphisms of curves and automorphisms of $\P^N$ in \cite{Silverman6}, and for morphisms of projective varieties by the author in \cite{Hutz1}.  The purpose of this article is to give a new proof of the nonexistence of poles.  The deformation argument used here leads to a simplified proof compared to \cite{Hutz1}, but lacks the detailed multiplicity information that allowed for the additional results presented there.

        We now state the problem precisely.  Let $K$ be an algebraically closed field and $X/K$ a projective variety of dimension $b$.  Let $\phi: X/K \to X/K$ be a morphism defined over $K$.  We can iterate the morphism $\phi$ and consider the resulting dynamical system.  As we will require tools from both dynamical systems and algebraic geometry in which the word cycle has two different meanings, we adopt the terminology: \emph{periodic cycle} to be the points in the orbit of a periodic point and \emph{algebraic zero-cycle} as a formal sum of points with integer multiplicities (only finitely many non-zero).  If all of the multiplicities of an algebraic zero-cycle are non-negative, we call it \emph{effective}.  For example, for $\phi \in K[z]$, if the algebraic zero-cycle of periodic points of period $n$ is effective, then the function $\phi^n(z)-z$ has no poles.  To generalize construction (\ref{eq1}) we follow \cite{Silverman6} and consider the graph of $\phi^n$ in the product variety $X \times X$ defined as
        \begin{equation*}
            \Gamma_n = \{(x,\phi^{n}(x)) \col x \in X\}
        \end{equation*}
        and the diagonal in $X \times X$ defined as
        \begin{equation*}
            \Delta = \{(x,x) \col x \in X\}.
        \end{equation*}
        Their intersection is precisely the periodic points of period $n$, and we can determine the multiplicity of points as the multiplicity of the intersection.  Denote the intersection multiplicity of $\Gamma_n$ and $\Delta$ at a point $(P,P) \in X \times X$ to be $a_P(n)$ and, when the intersection is proper, the algebraic zero-cycle of periodic points of period $n$ as
        \begin{equation*}
            \Phi_n(\phi) = \sum_{P \in X} a_P(n)(P).
        \end{equation*}
        Following construction (\ref{eq1}), define
        \begin{equation*}
            a_P^{\ast}(n) = \sum_{d \mid n} \mu\left(\frac{n}{d}\right) a_P(d)
        \end{equation*}
        and
        \begin{equation*}
            \Phi_n^{\ast}(\phi) = \sum_{d \mid n} \mu\left(\frac{n}{d}\right)
            \Phi_d(\phi) = \sum_{P \in X} a_{P}^{\ast}(n) (P),
        \end{equation*}
        where $\mu$ is the M\"{o}bius function defined above.
        \begin{defn}
            We call $\Phi^{\ast}_n(\phi)$ the \emph{$n\tth$ dynatomic cycle}\footnote{This term is inspired by ``cyclotomic'' much like ``Tribonacci'' was inspired by ``Fibonacci''.} and $a_P^{\ast}(n)$ the
            \emph{multiplicity} of $P$ in $\Phi^{\ast}_n(\phi)$.  If
            $a_P^{\ast}(n) >0$, then we call $P$ a periodic point of \emph{formal} period $n$.
        \end{defn}
        \begin{rem}
            In the one variable polynomial case, $\Phi^{\ast}_n$ is called a \emph{dynatomic polynomial} and has been studied extensively such as in \cite{Morton2, Morton, Silverman5}.  For a more complete background and additional references in this area see \cite{Silverman10}.
        \end{rem}
        \begin{defn}
            For $n \geq 1$, we say that $\phi^{n}$ is \emph{nondegenerate} if $\Delta$ and $\Gamma_n$
            intersect properly; in other words, if $\Delta \cap \Gamma_n$ is a finite set of points.
        \end{defn}
        \begin{rem}
            If $\phi^n$ is nondegenerate, then $\phi^d$ is nondegenerate for all $d \mid n$ since $\Delta \cap \Gamma_d \subseteq \Delta \cap \Gamma_n$.  Conversely, $\phi$ may be nondegenerate with $\phi^n$ degenerate, such as when $\phi$ is a nontrivial automorphism of a curve with finite order.
        \end{rem}
		We prove the following theorem.
        \begin{thm} \label{thm_main}
            Let $X \subset \mathbb{P}_K^{N}$ be a nonsingular, irreducible, projective variety defined over an algebraically closed field $K$ and let $\phi:X \to X$ be a morphism defined over $K$.  Let $P$ be a point in $X(K)$.  For all $n \geq 1$ such that $\phi^n$ is nondegenerate, $a_P^{\ast}(n) \geq 0$.
        \end{thm}

        Recall that we have defined $K$ to be an algebraically closed field, $X/K$ a projective variety of dimension $b$, and $\phi:X \to X$ a morphism defined over $K$. Let $P \in X(K)$ and let $R_P$ be the local ring of $X \times X$ at $(P,P)$ and let $I_{\Delta}, I_{\Gamma_n} \subset R_P$ be the ideals of the diagonal $\Delta$ and the graph of $\phi^n$ $\Gamma_n$, respectively.  We use Serre's definition of intersection multiplicity \cite[Appendix A]{Hartshorne},
        \begin{equation*}
            a_P(n) = i(\Delta,\Gamma_n;P) = \sum_{i=0}^{b -1} (-1)^{i}
            \dim_K(\tor_{i}(R_P/I_{\Delta},R_P/I_{\Gamma_n})).
        \end{equation*}
        In \cite{Hutz1} we first showed that $\tor_i(R_P/I_{\Delta},R_P/I_{\Gamma_n}) =0$ for $i \geq 1$ and then performed a detailed analysis of $\dim_K(\tor_0(R_P/I_{\Delta},R_P/I_{\Gamma_n}))$ which is the codimension of the ideal
        \begin{equation*}
           \dim_K \left(\tor_0(R_P/I_{\Delta},R_P/I_{\Gamma_n})\right) \cong \dim_K \left(R_P/I_{\Delta} \otimes R_P/I_{\Gamma_n}\right) \cong \dim_K \left(R_P/(I_{\Delta} + I_{\Gamma_n})\right).
        \end{equation*}
        The analysis included a detailed description of the coefficients under iteration of a system of multivariate power series.  In the present article, we instead deform a local representation of the map $\phi$ at each point $P \in \Delta \cap \Gamma_n$ and take the limit of the resulting algebraic zero-cycles.  Since effectivity is a local property, we need not consider the existence of a global deformation and are able to take quite simple deformations locally.  In the case $X=\P^N$, we could take a similarly simple global deformation since the deformed map will still be an endomorphism of the space.  The example in Section \ref{sect_proof} is such a case.

        To keep this article self-contained, before examining $\tor_0(R_P/I_{\Delta},R_P/I_{\Gamma_n})$ we first repeat the proofs from \cite{Hutz1} that $\Phi^{\ast}_n(\phi)$ is an algebraic zero-cycle, that $\tor_i(R_P/I_{\Delta},R_P/I_{\Gamma_n}) =0$ for $i \geq 1$, and that $a_P(n) \geq a_P(1)$ for all $n \in \N$.

    \section{Higher $\tor$-modules}
        Since $\phi^n$ is nondegenerate, $\Delta$ and $\Gamma_n$ intersect
        properly.  We also know $X \times X$ has dimension $2b$, $\Delta$ has dimension $b$, and
        $\Gamma_n$ has dimension $b$.  Consequently, $\Phi_n(\phi)$ is an algebraic zero-cycle.  Thus,
        $\Phi^{\ast}_n(\phi)$ is also an algebraic zero-cycle.

        \begin{lem}\textup{\cite[Corollary to Theorem V.B.4]{Serre}}\label{lem1}
            Let $(R,\mathfrak{m})$ be a regular local ring of dimension $b$, and let $M$ and $N$ be two
            non-zero finitely generated $R$-modules such that $M \otimes N$ is of finite
            length.  Then $\tor_i(M,N) = 0$ for all $i>0$ if and only if $M$ and $N$ are
            Cohen-Macaulay modules and $\dim M + \dim N = b$.
        \end{lem}

        \begin{thm} \label{thm10}
            Let $X$ be a nonsingular, irreducible, projective variety defined over a
            field $K$ and $\phi:X \to X$ a morphism defined over $K$ such that $\phi^n$
            is nondegenerate.  Let $P \in X(K)$.  Then, $\tor_i(R_P/I_{\Delta},R_P/I_{\Gamma_n}) =0$ for all $i>0$.
        \end{thm}

        \begin{proof}
            Let $b=\dim{X}$, then we have $\dim{X \times X} = 2b$ and $\dim{\Delta}=\dim{\Gamma_n}=b$.  The ideals $I_{\Delta}$ and $I_{\Gamma_n}$ are each generated by $b$ elements and $\Delta$ and $\Gamma_n$ intersect properly.  Therefore,
            \begin{equation*}
                \dim_K(R_P/(I_{\Delta} + I_{\Gamma_n})) = \text{length}(R_P/I_{\Delta} \otimes R_P/I_{\Gamma_n}) < \infty.
            \end{equation*}

            Thus, the union of the generators of $I_{\Delta}$ and the generators of $I_{\Gamma_n}$ are a system of parameters for $R_P$ \cite[Proposition III.B.6]{Serre}. Consequently, since the local ring $R_P$ is Cohen-Macaulay we can conclude that $R_P/I_{\Delta}$ is Cohen-Macaulay of dimension $b$ \cite[Corollary to Theorem IV.B.2]{Serre}.  Similarly with $I_{\Gamma_n}$, to conclude that $R_P/I_{\Gamma_n}$ is Cohen-Macaulay of
            dimension $b$.

            We have fulfilled the hypotheses of Lemma \ref{lem1} and can conclude the result.
        \end{proof}

        Recall that we can compute the codimension of an ideal from its leading term ideal:
        \begin{equation*}
            K[[X_1,\ldots,X_b]]/I \cong_K \Span(X^{v} \mid X^{v} \not\in LT(I)).
        \end{equation*}

        \begin{lem} \label{lem10}
            Assume $\phi^n$ is nondegenerate.  Then $a_P(n)\geq a_P(1)$ for all $n \in \N$.
        \end{lem}

        \begin{proof}
            If $P$ is not a fixed point of $\phi$, then the statement is trivial, so assume that $\phi(P)=P$.  It is clear that
            \begin{equation*}
               \Delta \cap \Gamma_1  \subseteq \Delta \cap \Gamma_n
            \end{equation*}
            and we have a local representation of $\phi = [\phi_1,\ldots,\phi_b]$ at the fixed point $P$.  Iterating this representation involves taking combinations of the $\phi_i$ and hence are all elements of the original ideal $I_{\Gamma_1}$.  Hence, we have
            \begin{equation*}
                I_{\Gamma_n} + I_{\Delta} \subseteq  I_{\Gamma_1} + I_{\Delta}.
            \end{equation*}
            Therefore,
            \begin{equation*}
                LT(I_{\Gamma_n} + I_{\Delta}) \subseteq LT(I_{\Gamma_1} + I_{\Delta})
            \end{equation*}
            which implies $a_P(n) \geq a_P(1)$.
        \end{proof}

\section{Proof of Theorem \ref{thm_main}} \label{sect_proof}
    The method is to deform the local intersection at $P$ into a flat family of algebraic zero-cycles whose generic member is a transverse intersection.  Then we take the limit of these algebraic zero-cycles to conclude the theorem.  The main references are flat families \cite[\S 6]{Eisenbud2}, families of algebraic cycles \cite[\S 10-11]{Fulton}, and basic analytic geometry in several complex variables \cite[\S1.4]{Dangelo}.

    Effecitivity is a local property,  so consider a point $P \in X(K)$.  In what follows, we work over the completion $\widehat{R_P}\cong K[[x_1,\ldots,x_b,y_1,\ldots,y_b]]=K[[\textbf{x},\textbf{y}]]$ so that we may consider our problem over a local power series ring.  Locally at $P$ we may write $\phi(\textbf{x})$ as a system of power series denoted $[\phi_1(\textbf{x}), \ldots, \phi_b(\textbf{x})]$.  We wish to deform an algebraic zero-cycle, so we consider $Z_{n,P}$ to be the algebraic zero-cycle obtained by intersecting the local equations for the diagonal $(x_1-y_1,\ldots,x_b-y_b)$ and the graph $(y_1-\phi_1^n(\textbf{x}), \ldots, y_b-\phi_b^n(\textbf{x}))$ as analytic varieties.  We have now reduced the problem to a multiplicity question on an analytic variety $Z_{n,P}$ and we can deform in this situation without concern for the global structure on $X$.  We deform $Z_{n,P}$ by considering the iterates of \[\phi(\textbf{x},t)=[\phi_1(\textbf{x})+t,\ldots,\phi_b(\textbf{x})+t]\]
    for a parameter $t \in \A^1_K$ and their graphs denoted $\Gamma_n(t)$.  Note that $\phi^n(\textbf{x},0)=\phi^n(\textbf{x})$.   We denote the deformed family as $Z_{n,P}(t)$.  Notice that we are deforming and then iterating so that $Z_{n,P}(t)$ is associated to $(\phi(\textbf{x},t))^n$.  It is important that $P$ be a fixed point so that a local representation of $\phi^n$ at $P$ is given by the iterate of the local representation of $\phi$ at $P$. Before beginning the proof, we illustrate the method with the one dimensional polynomial example mentioned in the introduction.
    \begin{exmp}
        Consider $\phi(x) = x^2-3/4$.  The fixed points are determined as
        \begin{equation*}
            \phi(x) - x = (x+1/2)(x-3/2) =0
        \end{equation*}
        and are both multiplicity one, $a_{-1/2}(1) =a_{3/2}(1)=1$.
        After deforming, the fixed points as functions of $t$ are determined as
        \begin{equation*}
            \phi(x,t) -x = x^2 - x -3/4 + t =0
        \end{equation*}
        with two distinct multiplicity one solutions, unless $t=1$,
        \begin{equation*}
            P_1(t) = \frac{1}{2}(1+2\sqrt{1-t}) \quad \text{and} \quad P_2(t) = \frac{1}{2}(1-2\sqrt{1-t}).
        \end{equation*}
        The $2$-periodic points are determined as
        \begin{equation} \label{eq4}
            \phi^2(x) - x = (x+1/2)^3(x-3/2)=0.
        \end{equation}
        The solutions are not distinct and they are both fixed points.  In particular, $a_{-1/2}(2) = 3$ and $a_{3/2}(2)=1$.  This causes $a_{-1/2}^{\ast}(2) = 2$ even though $-\frac{1}{2}$ is a fixed point.  It is this higher multiplicity counting that makes the statement of effectivity interesting and the complication that this deformation argument seeks to circumvent.  After deforming, the $2$-periodic points are determined as
        \begin{equation*}
            \phi^2(x,t)-x = x^4 + (2t - 3/2)x^2 -x + (t^2 - 1/2t - 3/16)=0
        \end{equation*}
        with four distinct multiplicity one solutions, unless $t=1$ or $t=0$,
        \begin{align*}
            P_1(t) &= \frac{1}{2}(1+2\sqrt{1-t}), \qquad P_2(t) = \frac{1}{2}(1-2\sqrt{1-t}),\\
            P_3(t) &= -\frac{1}{2}(1-2\sqrt{-t}), \quad \text{and} \quad P_4(t) = -\frac{1}{2}(1+2\sqrt{-t}).
        \end{align*}
        There are two main facts to notice about these four points.  First, as $t \to 0$, the points correspond in multiplicity to the undeformed system (\ref{eq4})
        \begin{equation*}
            P_1(0) = \frac{3}{2} \qquad P_2(0)=P_3(0)=P_4(0)=-\frac{1}{2}.
        \end{equation*}
        Secondly, $P_1(t)$ and $P_2(t)$ are fixed points, while $P_3(t)$ and $P_4(t)$ are periodic points with minimal period $2$ and all four occur with multiplicity one.  Examining the dynatomic multiplicities we have
        \begin{align*}
            a_{P_1(t)}^{\ast}(2) &= a_{P_2(t)}^{\ast}(2)=0\\
            a_{P_3(t)}^{\ast}(2) &= a_{P_4(t)}^{\ast}(2) = 1
        \end{align*}
        for a total of
        \begin{equation*}
            a_{-1/2}^{\ast}(2) = a_{P_2(t)}^{\ast}(2) +  a_{P_3(t)}^{\ast}(2) + a_{P_4(t)}^{\ast}(2) = 2
        \end{equation*}
        as computed directly above.
    \end{exmp}
    We start by showing that $Z_{n,P}(t)$ is a flat family using the following local criteria for flatness.
    \begin{lem}\textup{\cite[Corollary 6.9]{Eisenbud2}} \label{lem_flatness}
        Suppose that $(R,\mathfrak{m})$ is a local Noetherian ring.  Let $x \in R$ be a non-zero divisor on $R$ and let $M$ be a finitely generated $R$-module.  If $x$ is a non-zero divisor on $M$, then $M$ is flat over $R$ if and only if $M/xM$ is flat over $R/(x)$.
    \end{lem}
    \begin{prop} \label{prop_flat}
        Let $n \in \N$ be such that $\phi^n$ is non-degenerate and let $P \in X(K)$.  The family $Z_{n,P}(t)$ is flat over $K[[t]]$.
    \end{prop}
    \begin{proof}
        Recall that $Z_{n,P} = a_{P}(n)(P)$ with $a_P(n) = \dim_K \widehat{R}_P/(I_{\Gamma_n} + I_{\Delta})$.  Thus, to show flatness for $Z_{n,P}(t)$, we need to show flatness for $\widehat{R}_P[[t]]/(I_{\Gamma_n(t)} + I_{\Delta})$.

        We apply Lemma \ref{lem_flatness} with $M = \widehat{R}_P[[t]]/(I_{\Gamma_n(t)} + I_{\Delta})$, $R = K[[t]]$, and $x=t$.

        We see that $M/tM \cong \widehat{R}_P/(I_{\Gamma_n} + I_{\Delta})$ from our choice of deformation and $K[[t]]/(t) \cong K$.  Thus, $M/tM$ is a flat $K$-module since it is a finite dimensional $K$-vector space by the nondegeneracy of $\phi^n$.  Now, we just need to show that $t$ is not a zero divisor on $M$.

        Assume that $t$ is a zero divisor.  Then, there exists an $m \in M$ with $m \neq 0$ such that $tm =0$.  In particular there exist $a_i \in K[[t]]$ such that \[tm = \sum_{i=1}^{2b} a_i m_i,\] where $m_i$ are the generators of $(I_{\Gamma_n(t)} + I_{\Delta})$.  Specializing to $t=0$, we must have \[\left(\sum_{i=1}^{2b} a_i m_i\right)_{t=0} = 0,\] with $(m_i)_{t=0} \neq 0$ for all $i$.  Assume that $(a_i)_{t=0} = 0$ for all $i$, then we have
         \[
            \sum_{i=1}^{2b} \frac{a_i}{t}m_i = m
         \]
         with $\frac{a_i}{t} \in K[[t]]$.  This contradicts $m \not\in (I_{\Gamma_n(t)} + I_{\Delta})$.  So we have at least one $(a_i)_{t=0} \neq 0$ and, hence, there is a relation among the $(m_i)_{t=0}$, which contradicts the assumption that $\phi^n$ is non-degenerate.
    \end{proof}

    Recall that $P$ is a (possibly high multiplicity) solution to the set of equations $\phi_i(\textbf{x}) = x_i$.  We have perturbed this system to obtain the set of equations $\phi_i(\textbf{x},t) = x_i$ for $i=1,\ldots, b$.  We will use the Weierstrass Preparation Theorem to obtain distinct solutions $P_j(t)$ as power series in $t$ for some set of $j=1,\ldots,a_P(n).$

    \begin{proof}[Proof of Theorem \ref{thm_main}]
        We fix $n$ and consider each point $P \in X(K)$.  If $P$ is not periodic of period $n$, then we have $a_d(P) = 0$ for all $d \mid n$ and hence $a_P^{\ast}(n) = 0$.  So we may assume that $P$ is periodic of some minimal period $m \mid n$.  Replacing $P$ with $\phi^m(P)$, $\phi$ with $\phi^m$, and $n$ with $n/m$ we may assume that $P$ is a fixed point of $\phi$.  Working locally, we consider the family of algebraic zero-cycles $Z_{n,P}(t)$ defined above.  By Proposition \ref{prop_flat} this is a flat family and, thus, by \cite{Lazarsfeld} we have that
        \[ \lim_{t \to 0} Z_{n,P}(t) = Z_{n,P}(0) = Z_{n,P}. \]
        In particular, if $P_j(t)$ are the points in the support of $Z_{n,P}(t)$ which go to $P$ as $t \to 0$, then if we write the algebraic zero-cycle as
        \[ Z_{n,P}(t) = \sum_j a_{P_j(t)}(n) (P_j(t)), \]
        we have that
        \[ a_P(n) = \sum_j a_{P_j(t)}(n) \quad \text{and} \quad a_P^{\ast}(n) = \sum_j a_{P_j(t)}^{\ast}(n). \]
        Note that each $P_j(t)$ is periodic with minimal period dividing $n$ and there are finitely many such $P_j(t)$, in fact by flatness, there are $a_P(n)$ of them counted with multiplicity.  From standard results in the theory of analytic varieties in several complex variables concerning the Weierstrass Preparation theorem and multiple roots of Weierstrass polynomials \cite[\S 1.4]{Dangelo}, we know that the set of $t$ values for which there is a solution $P_j(t)$ with multiplicity greater than one is a thin set.  In particular, generically there are $a_P(n)$ distinct $P_j(t)$ which satisfy $P_j(0)=P$.  Finally, $a_P(d) \leq a_P(n)$ for $1 \leq d \leq n$ by Lemma \ref{lem10}.  Thus, by avoiding a thin set of $t$, for each $d \mid n$ each $P_j(t)$ occurs with multiplicity $1$ in $Z_{d,P}(t)$ if it has minimal period dividing $d$ and multiplicity $0$ otherwise.  Using properties of the M\"obius function, we compute for $P_j(t)$ with minimal period $m < n$
        \begin{equation*}
            a_{P_j(t)}^{\ast}(n) = \sum_{d \mid n} \mu\left(\frac{n}{d}\right) a_{P_j(t)}(d) = \sum_{d \mid \frac{n}{m}} \left(\frac{n}{dm}\right) a_{P_j(t)}(d) = \sum_{d \mid \frac{n}{m}} \mu\left(\frac{n}{dm}\right) = 0.
        \end{equation*}
        For $P_j(t)$ with minimal period $n$ we compute
        \begin{equation*}
            a_{P_j(t)}^{\ast}(n) = \sum_{d \mid n} \mu\left(\frac{n}{d}\right) a_{P_j(t)}(d) = a_{P_j(t)}(n) = 1.
        \end{equation*}
        Thus, any $P_j(t)$ with minimal period strictly less than $n$ must contribute $0$ to $a_P^{\ast}(n)$ and any $P_j(t)$ with minimal period $n$, must contribute positively.  In particular, $a_P^{\ast}(n) \geq 0$.
    \end{proof}

\providecommand\biburl[1]{\texttt{#1}}

\end{document}